\documentclass[11pt]{amsart}
\usepackage{amsmath, amsthm}
\usepackage{amssymb}
\usepackage{amscd}
\usepackage{url}
\usepackage{graphicx}
\usepackage{color}

\numberwithin{equation}{section}

\theoremstyle{plain}

\newtheorem{theorem}[subsection]{Theorem}
\newtheorem{proposition}[subsection]{Proposition}
\newtheorem{lemma}[subsection]{Lemma}
\newtheorem{corollary}[subsection]{Corollary}

\theoremstyle{definition}

\newcommand{\Q}{\mathbb{Q}}
\newcommand{\Z}{\mathbb Z}

\newcommand{\R}{\mathbb{R}}

\title{Lattice points in elliptic paraboloids}
\author{Fernando Chamizo}
\author{Carlos Pastor}
\thanks{The first author is partially supported by the MINECO grant MTM2014-56350-P. 
The second author has been supported by the ''la Caixa''-Severo Ochoa international PhD programme at the Instituto de Ciencias Matem\'aticas (CSIC-UAM-UC3M-UCM)}

\begin{document}

\begin{abstract}
We consider the lattice point problem corresponding to a family of elliptic paraboloids in $\R^d$ with $d\ge3$ and we prove the expected to be optimal exponent, improving previous results. 
This is especially noticeable for $d=3$ because the optimal exponent is conjectural even for the sphere.
We also treat some aspects of the case $d=2$, getting for a simple parabolic region an $\Omega$-result that is unknown for the classical circle  and divisor problems.
\end{abstract}

\maketitle

\pagestyle{plain} 

\section{Introduction}

Given a fixed compact subset $\mathcal{K}$ of $\mathbb{R}^d$ satisfying certain regularity conditions, the fundamental problem in lattice point theory consists in estimating
\begin{equation}\label{N_R}
\mathcal{N}(R)=
\#\big\{
\vec{n}\in\Z^d\;:\;
R^{-1}\vec{n}\in\mathcal{K}
\big\}
\end{equation}
for large values of $R$. The natural approximation is 
$|\mathcal{K}|R^d$, where
$|\mathcal{K}|$ stands for the volume of $\mathcal{K}$, and one is interested in the optimal exponent
\begin{equation}\label{alphak}
\alpha_\mathcal{K}
=
\inf
\big\{
\alpha>0\;:\; 
\mathcal{N}(R)=|\mathcal{K}|R^d+O\big(R^\alpha\big)
\big\}.
\end{equation}
For instance, when $\mathcal{K}$ is a convex body with smooth boundary and non-vanishing Gaussian curvature it is known $\alpha_{\mathcal{K}} \leq 131/208$ for $d = 2$, and $\alpha_{\mathcal{K}} \leq d - 2 + r(d)$ with $r(d) = 73/158$ for $d = 3$ and $r(d) = (d^2+3d+8)/(d^3 + d^2 +5d + 4)$ for $d \geq 4$ \cite{huxley, guo}.

Upper bounds for $\alpha_{\mathcal{K}}$ have been extensively studied for a great number of families of sets $\mathcal{K}$, but sharp results are scarce. Paradigmatic to the theory are the celebrated problems of Gauss and Dirichlet, dealing with the circle and a hyperbolic region of the plane (with extra conditions in the latter case to ensure compactness). The best known upper bound for both of them (also valid for ellipses) is the aforementioned $\alpha_{\mathcal{K}} \leq 131/208$, still far from Hardy's conjecture $\alpha_{\mathcal{K}} = 1/2$. The situation is not much better understood in three dimensions, where the best known upper bound for the sphere and the average of the class number (corresponding to a family of hyperboloids) is $\alpha_{\mathcal{K}} \leq 21/16$ \cite{HeBr, ChIw}, result extended to rational ellipsoids in \cite{ChCrUb}. The expected value of $\alpha_{\mathcal{K}}$ in these cases, and in general for $d \geq 3$, is $d-2$.

As the dimension increases the lattice point problem for the ball and rational ellipsoids becomes simpler, due to the higher regularity of the number of representations of integers by rational quadratic forms. This leads in a fairly easy manner to the sharp result $\alpha_{\mathcal{K}} = d-2$ for $d \geq 4$, contained in classical works. Less is known in the irrational case, where the inequality $\alpha_{\mathcal{K}} \leq d-2$ was finally achieved in \cite{BeGo1} for $d \geq 9$ and later in \cite{gotze} for $d \geq 5$. Surprisingly, the error term is, in contrast with the rational case, $o\big( R^{d-2} \big)$, and this led to a proof the Davenport-Lewis conjecture about the gaps of the image of $\mathbb{Z}^d$ under irrational quadratic forms \cite{BeGo2}. For further reference on these and other lattice point problems related to number theory questions we refer the reader to~\cite{IvKr}.

It is surprising the little attention attracted by the remaining conic, the parabola, and its higher-dimensional analogues. The bidimensional case was addressed by Popov \cite{popov}, proving the upper bound $\alpha_{\mathcal{K}} \leq 1/2$ for a variant of the lattice point problem $\mathcal{K} = \big\{ |y| \leq c - (x + \beta)^2 \big\}$, $c > 0$, and the sharp result $\alpha_{\mathcal{K}} = 1/2$ when $\beta = 0$ and $c \in \mathbb{Q}$. This remarkable difference with the Gauss and Dirichlet problems is a consequence of the fact that the error term for the parabola can be expressed in terms of $1$-dimensional quadratic exponential sums, and these can be finely estimated with simple Diophantine considerations. In fact in some particular cases this leads, via the evaluation of Gauss sums, to an explicit formula for the error term involving 
$L$-functions that seems to have been overlooked in the literature. We derive this formula for the case $c = 1, \beta = 0$ and show that it can be used to obtain a one-sided $\Omega$-result beyond what is known for both the circle and the hyperbola \cite{hafner, sound}, namely (see Proposition~\ref{omega2})
\[
 \mathcal{N}(R) - |\mathcal{K}|R^2 
 = 
 \Omega_{-}\big(
 R^{1/2}\exp(c\sqrt{\log R}/\log\log R)
 \big) 
 \qquad \text{for any } 
 \quad c < \sqrt{2}.
\]

In higher dimensions the natural set to study is that of elliptic paraboloids of the form
\begin{equation}\label{main_P}
\mathcal{P} = \big\{
(\vec{x},y)\in\R^{d-1}\times \R
\;:\;
|y|\le c-Q(\vec{x}+\vec{\beta})
\big\},
\end{equation}
where $Q$ is a positive definite quadratic form, $\vec{\beta}$ is a fixed vector in $\mathbb{R}^{d-1}$ and $c$ a positive constant. The particular case $\vec{\beta} = 0$ was considered in a slightly different form by Kr\"atzel \cite{kratzel1, kratzel2}, but his method only yielded the inequality $\alpha_{\mathcal{P}} \leq d-2$ under the strong assumptions $d \geq 5$ and $Q$ either rational or diagonal (proving, in the rational case, $\alpha_{\mathcal{P}} = d-2$ for $c \in \mathbb{Q}$). Partial results were given under weaker rationality assumptions in terms of the coefficient matrix $A = (a_{ij})$ of $Q$. In particular, Kr\"atzel obtained $\alpha_{\mathcal{P}} \leq d-5/3$ for $d \geq 3$ as long as $a_{12}/a_{11}, a_{22}/a_{11} \in \mathbb{Q}$. We improve these results:
\begin{theorem}\label{main}
If $a_{12}/a_{11}, a_{22}/a_{11} \in \mathbb{Q}$ then the inequality $\alpha_{\mathcal{P}} \leq d-2$ holds for any $d \geq 3$. If moreover $\vec{\beta} = 0$, $c \in \mathbb{Q}$ and $Q$ is rational then this is sharp.
\end{theorem}
Note that no assumptions are imposed on the remaining coefficients, and therefore this result extends the upper bound $\alpha_{\mathcal{P}} \leq d-2$ not only to $d = 3, 4$ and $\vec{\beta} \neq 0$, but also to a wider family of higher-dimensional paraboloids. The key step in the proof is the estimation of a certain quadratic exponential sum in two variables, which is done employing what can be considered a toy version of the circle method. Bounds this precise are out of reach for the exponential sums arising in most lattice point problems, and this accounts for the striking difference between our theorem and what is currently known for ellipsoids and hyperboloids. In fact, to the best of our knowledge, Theorem~\ref{main} constitutes the first sharp result for a lattice point problem in three dimensions.

\

The structure of the paper is as follows: First we revisit the two-dimensional case, deriving the exact formula for the error term. After this we devote \S3 to estimate the quadratic exponential sum involved in the proof of the first part of Theorem~\ref{main}, and \S4 to the proof itself. Finally in \S5 we prove some $\Omega$-results that readily imply the second part of Theorem~\ref{main}, together with more precise two-sided $\Omega$-results for the parabolic region considered in \S2.

\

Along this paper, $e(x)$ is an abbreviation for $e^{2\pi i x}$ and $\epsilon$ denotes an arbitrarily small positive quantity,  that may change in value at each appearance. We employ $f(x)=O\big(g(x)\big)$ and $f(x)\ll g(x)$ indistinctly to mean that $\limsup |f|/|g|<\infty$ and $f(x)=o\big(g(x)\big)$ meaning $\lim f/g=0$. The negation of the latter is denoted by $f(x) = \Omega\big(g(x)\big)$, while the symbol $\Omega_+$ (or $\Omega_-$) is employed to specify that the positive (or negative) part of $f(x)$ is $\Omega\big(g(x)\big)$.

\section{A parabolic region}

Popov gave in \cite{popov} an asymptotic for the number of lattice points under a parabola, i.e. in the region $\big\{0 \leq y \leq ax^2 \leq c\big\}$ of the XY plane. His method readily applies to the region of the form $\big\{ |y| \leq c - (x + \beta)^2 \big\}$ mentioned in the introduction, yielding $\alpha_{\mathcal{K}} \leq 1/2$. In fact it is possible to give a very short proof of this result in few lines appealing to classic estimates of quadratic sums \cite[Th.8.11]{IwKo}. 
When the dilation is integral these quadratic sums can be evaluated explicitly, and the resulting formula relates the error term for this lattice point problem to the class number associated to a family of imaginary quadratic fields.
We credit Professor Antonio C\'{o}rdoba for pointing out the relation with the class number in the early 90's while he was the Ph.D. advisor of the first named author. 

\

To illustrate the situation, we consider in this section a simple parabolic region
\[
 \mathcal{P}_2=
 \big\{
 (x,y)\in\R^2\;:\; |y| \le 1-x^2
 \big\},
\]
and denote by $\mathcal{N}_2(R)$ the number of lattice points in $\mathcal{P}_2$ scaled by $R$.

\begin{theorem}\label{int_para}
Let $N$ be an odd positive integer and let $N^*$ be the greatest square dividing $N$. Then
\[
 \mathcal{N}_2(N)=|\mathcal{P}_2|N^2
 +
 \frac 13
 +2\sqrt{N^*}
 -\frac{4}{\pi}
 \sum_{\substack{d\mid N\\ d\equiv3\ (4)}}
 \sqrt{d}L(1,\chi_{-d})
\]
where $L(1,\chi_{-d})$ is the $L$-function corresponding to the Kronecker symbol
$\chi_{-d}=\Big(\frac{-d}{\cdot}\Big)$. 
\end{theorem}

With some effort the result can be extended, with modifications, to cover the even case. 

\

Two particular cases of Theorem~\ref{int_para} deserve special attention. They will be used in~\S5 to obtain one-sided $\Omega$-results for this lattice point problem.

\begin{corollary}\label{prime4k1}
If the prime factors of $N$ are of the form $4k+1$, then  
\[
 \mathcal{N}_2(N)-|\mathcal{P}_2|N^2
 =
 \frac 13
 +2\sqrt{N^*}.
\]
\end{corollary}

\begin{corollary}\label{classn}
If $N$ is squarefree then  
\[
 \mathcal{N}_2(N)-|\mathcal{P}_2|N^2
 =
 \frac 73
 -4
 \sum_{\substack{d\mid N\\ d\equiv3\ (4)}}
 \omega_d h(-d)
\]
where $h(-d)$ is the class number of the integer ring of $\Q(\sqrt{-d})$ and $\omega_d=1$ except for $\omega_{3}=1/3$. 
\end{corollary}

\begin{proof}
Apply Dirichlet class number formula in  Theorem~\ref{int_para} for the fundamental discriminant~$-d$.
\end{proof}

\begin{proof}[Proof of Theorem~\ref{int_para}]
Writing $\psi(x)=x-\lfloor x\rfloor -1/2$,
\[
 \mathcal{N}_2(N)
 =
 2
 \sum_{n=-N}^N
 \Big(
 N-\frac{n^2}{N}
 \Big)
 -2
 \sum_{n=-N}^N
 \psi\Big(
 -\frac{n^2}{N}
 \Big).
\]
The first sum is $(4N^2-1)/3$ and the area is clearly $|\mathcal{P}_2|=8/3$. Then
\[
 \mathcal{N}_2(N)-|\mathcal{P}_2|N^2
 =
 -\frac 23
 -2
 \sum_{n=-N}^N
 \psi\Big(
 -\frac{n^2}{N}
 \Big)
 =
 \frac 13
 -4
 \sum_{n=1}^N
 \psi\Big(
 -\frac{n^2}{N}
 \Big).
\]
We now substitute
\[
 \psi(x)
 =
 \Im
 \sum_{m=1}^\infty
 \frac{e(-mx)}{\pi m}
 +\begin{cases}
   0 &\text{if}\quad x\not\in\Z,
   \\
   -1/2 &\text{if}\quad x\in\Z.
  \end{cases}
\]
Note that $N$ divides $n^2$ exactly $\sqrt{N^*}$ times in the range $1\le n\le N$, and hence
\begin{equation}\label{aux_para}
 \mathcal{N}_2(N)-|\mathcal{P}_2|N^2
 =
 \frac 13
 +2\sqrt{N^*}
 -\frac{4}{\pi}
 \sum_{m=1}^\infty
 \frac{1}{m}
 \Im G(m;N)
\end{equation}
where $G(m;N)$ is the quadratic Gauss sum $\sum_{n=1}^N e\big(mn^2/N\big)$. Let $d_m=N/\gcd(m,N)$, the evaluation of $\Im G(m;N)$ reads \cite{IwKo}
\[
 \Im G(m;N)
 =
 \begin{cases}
  0&\text{if}\quad d_m\equiv 1\pmod{4},
  \\
  \frac{N}{\sqrt{d_m}}
  \Big(
  \frac{md_m/N}{d_m}
  \Big)&\text{if}\quad d_m\equiv 3\pmod{4}.
 \end{cases}
\]
Substituting in \eqref{aux_para} and noting that when $d_m$ is fixed $md_m/N$ runs over all positive integers coprime to $d_m$, we have 
\[
 \mathcal{N}_2(N)-|\mathcal{P}_2|N^2
 =
 \frac 13
 +2\sqrt{N^*}
 -\frac{4}{\pi}
 \sum_{\substack{d\mid N\\ d\equiv3\ (4)}}
 \sqrt{d}
 \sum_{m=1}^\infty
 \frac{1}{m}
 \Big(\frac{m}{d}\Big).
\]
By the quadratic reciprocity law for the Jacobi-Kronecker symbol \cite[\S3.5]{IwKo}, the innermost sum equals $L(1,\chi_{-d})$. 
\end{proof}

\section{Elliptical summation}

Consider the well known Hardy-Littlewood bound\footnote{Although this result is implicit in the work \cite{HaLi}, we refer the reader to \cite{fiedler} for a closer statement.}
\begin{equation}\label{hali}
\sum_{n=-N}^N e(n^2 x)\ll q^{-1/2}N
  \qquad\text{if }\quad
 \Big|x-\frac aq\Big|\le \frac{1}{qN}
 \quad\text{with}\quad q\le N.
\end{equation}
Squaring this formula we have
\[
\mathop{\sum\!\sum}_{(n,m)\in \mathcal{C}}
e\big((n^2+m^2) x\big)\ll q^{-1}N^2
  \qquad\text{with }\quad
  \mathcal{C}=[-N,N]\times [-N,N].
\]
In principle it is not clear whether the square $\mathcal{C}$ can be replaced by a circle or ellipse. This is forced in our approach (with extra linear terms) to keep the symmetry. Proposition~\ref{circular} below shows that this is possible losing at most a power of logarithm.
The problem was addressed in \cite{kratzel1} and \cite{kratzel2} via the estimation of 2-dimensional exponential sums getting a weaker result, and recently in \cite{HeHu} when the sum has a certain smooth cut-off.
Here we employ a simplified version of the circle method, taking advantage of the fact that only upper bounds are required on any arc.

For convenience, instead of the condition in \eqref{hali} we consider the Farey dissection of the continuum 
\[
\R= \bigcup_{a/q}\mathcal{A}_{a/q}
\qquad\text{with }\quad
\mathcal{A}_{a/q}=
\Big[
\frac{a+a^-}{q+q^-}
,
\frac{a+a^+}{q+q^+}
\Big)
\]
where $a^-/q^-<a/q<a^+/q^+$ are consecutive fractions in the Farey sequence of a fixed order, extended periodically. In our case we take the order to be $\lfloor N^{1/2}\rfloor$. In this way we can assign to each~$x$  a unique $a_x/q_x$ such that
\begin{equation}\label{farey_a}
x\in 
\mathcal{A}_{a_x/q_x}
\qquad\text{with }\quad
q_x\le N^{1/2}.
\end{equation}

\begin{proposition}\label{circular}
Let $Q$ be an integral positive definite binary quadratic form and $\alpha,\beta$ arbitrary real numbers. Then for every $N\ge 2$ and $x\in\R$ satisfying \eqref{farey_a}, we have
\[
\sum_{0\le n\le N} r_{\alpha,\beta}(n) e(n x)\ll 
\frac{N(\log N)^2}{q_x+N|q_xx-a_x|}
\qquad
\text{where \quad}
r_{\alpha,\beta}(n) = \sum_{Q(n_1,n_2)=n} e(\alpha n_1+\beta n_2),
\]
uniformly in $\alpha$ and $\beta$.
\end{proposition}
In what follows we introduce and estimate some auxiliary functions that will be used in the proof of Proposition~\ref{circular}. To simplify the notation all the bounds will be expressed in terms of the $1$-periodic function
\[
 B(t)=\min(N, \|t\|^{-1})
\]
where in this section $\|\cdot\|$ means the distance to the nearest integer. 

\begin{lemma}\label{b_dir}
For $N\in\Z^+$ 
\begin{equation}\label{dir_b}
\sum_{0\le n\le N}
e(nt) e^{2\pi n/N}
\ll B(t).
\end{equation}
\end{lemma}
\begin{proof}
The left hand side is a geometric series which can be readily bounded.
\end{proof}

Let $A$ be the matrix of the integral quadratic form $Q$ and consider
\[
 \theta_{\vec{v}}(z) 
 = 
 \sum_{n\ge 0} r_{\alpha,\beta}(n) e(nz)
 \qquad\text{with }\quad
 \vec{v}=\frac 12 A^{-1}
 \begin{pmatrix}
  \alpha\\ \beta
 \end{pmatrix}.
\]
This holomorphic function in the upper half plane corresponds to a Jacobi modular form for some special values of $\alpha$ and $\beta$. The reason to parametrize it in terms of $\vec{v}$ is to make the transformation formula, which we state next, as concise as possible. The proof is adapted from \cite[Ch.4]{siegel}, where it is presented in the more general context of indefinite forms.
\begin{lemma}\label{trans_f}
If $z$ and $w$, in the upper half plane, are related by a modular transformation
\[
 w=\frac{az+b}{cz+d}
 \qquad\text{with }\quad
 \gamma:=
 \begin{pmatrix}
  a&b\\ c&d
 \end{pmatrix}
 \in\text{SL}_2(\Z)
 \quad\text{ and }\quad c\ne 0, 
\]
then 
\[
 j_\gamma(z)
 \theta_{\vec{v}}(z) 
 =
 \frac{\delta(\gamma,\vec{v})}
 {2\sqrt{\det{A}}}
 \sum_{\vec{l}\in\mathcal{L}}
 G_{\vec{l}}
 \sum_{\vec{x}\in\Z^2+\vec{l}}
 e\big(
 wQ(\vec{x}+c\vec{v})
 -2a\vec{x}\cdot A\vec{v}
 \big),
\]
where $\delta$ is a certain function with $|\delta|=1$, $j_\gamma(z)= cz+d$, as usual, and $G_{\vec{l}}$ is a normalized Gauss sum associated to each representative of the quotient of lattices $\mathcal{L}=\frac 12A^{-1}\Z^2/\Z^2$, namely
\[
 G_{\vec{l}}
 =
 \frac 1c
 \sum_{\vec{g}\ (\text{\rm mod}\;{c})}
 e\Big(
 -\frac{aQ(\vec{l}+d\vec{g})}{c}
 \Big).
\]
\end{lemma}

\begin{proof}
By the definition of $\theta_{\vec{v}}(z)$ and separating the classes modulo $c$, 
\[
 \theta_{\vec{v}}(z)
 =
 \sum_{\vec{n}\in\Z^2}
 e\big( zQ(\vec{n})+2\vec{n}\cdot A\vec{v}\big) 
 =
 \sum_{\vec{g}\ (\text{\rm mod}\;{c})}
 \sum_{\vec{m}\in\Z^2}
 e\big( zQ(c\vec{m}+\vec{g})+2(c\vec{m}+\vec{g})\cdot A\vec{v}
 \big).
\]
Writing $\big(j_\gamma(z)-d\big)/c$ instead of $z$ and completing squares, the phase can be expressed as $P_1 +P_2$ with 
\[
 P_1= \frac{j_\gamma(z)}{c}
 Q\Big(
 c\vec{m}+\vec{g}+\frac{c\vec{v}}{j_\gamma(z)}
 \Big)
 \qquad\text{and}\qquad
 P_2= -\frac{c}{j_\gamma(z)}Q(\vec{v})
 -\frac dc Q( c\vec{m}+\vec{g} ). 
\]
Note that $P_2$ does not change modulo 1 when  $\vec{m}$ varies  and we can put $\vec{m}=\vec{0}$. On the other hand, by Proposition~10.1 of \cite{iwaniec}, 
\[
 \sum_{\vec{m}\in\Z^2}
 e\big( 
 P_1
 \big) 
 =
 \frac{i(\det A)^{-1/2}}{2c j_\gamma(z)}
 \sum_{\vec{m}\in\Z^2}
 e\Big(
 -\frac{Q(A^{-1}\vec{m}/2)}{c j_\gamma(z)}
 +
 c^{-1}\big(\vec{g}+\frac{c\vec{v}}{j_\gamma(z)}\big)\cdot\vec{m}
 \Big).
\]
Under the change of variables $\vec{x}= \frac 12 A^{-1}(-\vec{m})$ with $\vec{x}=\vec{n}+\vec{l}$, where $\vec{l}\in\mathcal{L}$ and $\vec{n}\in\Z^2$, this phase corresponds to
\[
 P_3
 =
 -\frac{Q(\vec{x}) +2c\vec{v}\cdot A\vec{x}}{c j_\gamma(z)}
 -\frac{2}{c}\vec{g}\cdot A\vec{x}.
\]
Then, substituting $\big(j_\gamma(z)\big)^{-1}=-cw + a$ in $P_2$ and $P_3$,
\[
 e(P_2+P_3)=
 e\big(
 wQ(\vec{x})+2(cw-a)\vec{v}\cdot A\vec{x}
 +c(cw-a)Q(\vec{v})
 \big)
 e\big(
 -\frac ac Q(\vec{x})
 -\frac 2c \vec{x}\cdot A\vec{g}
 -\frac dc Q(\vec{g})
 \big).
\]
The last exponential is $e\big(-aQ(\vec{x}+d\vec{g})/c\big)$ because $ad\equiv 1\pmod{c}$, and when we sum on $\vec{g}$ we obtain $cG_{\vec{l}}$. It only remains to note that the argument of the first exponential can be written as
$wQ(\vec{x}+c\vec{v})-2a\vec{v}\cdot A\vec{x}-acQ(\vec{v})$.
\end{proof}

\begin{lemma}\label{b_theta}
With the previous notation and $x$ as in \eqref{farey_a},
\[
\theta_{\vec{v}}(x+i/N)
\ll q_x^{-1}B(x-a_x/q_x),
\]
uniformly in $\vec{v}$.
\end{lemma}
\begin{proof}
Take in Lemma~\ref{trans_f} a matrix $\gamma$ such that $\gamma(p/q)=\infty$. Since $\Im w=\Im z/|j_\gamma(z)|^2$ with $|j_\gamma(z)|=|qz-p|$ and $G_{\vec{l}}\ll 1$ \cite[Lemma 20.12]{IwKo}, the right hand side of the transformation formula can be estimated term-wise and the exponential decay implies
$\theta_{\vec{v}}(z)\ll |qz-p|^{-1}$ uniformly in $\vec{v}$ when 
$|qz-p|^2=O(\Im z)$.
Choosing now $z=x+i/N$ and $p/q=a_x/q_x$ it is enough to note that
$|t+i/N|^{-1}\ll B(t)$ for $|t|\le 1/2$. 
\end{proof}

\begin{lemma}\label{BcB}
For $t\in\R$ we have 
\[
(B*B)(t)
:=
\int_{-1/2}^{1/2}
B(u)B(t-u)\; du
\ll
N
\frac{\log(2+N\|t\|)}{2+N\|t\|}.
\]
\end{lemma}

\begin{proof}
Cauchy's inequality gives 
$(B*B)(t)\ll \int_0^1 |B|^2\ll N$. Using this and the symmetry, we can assume 
$2N^{-1}<t<1/2$. 
If $0<u<1/2$ it is clear that the distance from $t$ to $u$ is smaller than the distance from $t$ to $-u$. Hence $B(t-u)\ge B(t+u)$ and
$(B*B)(t)\le 2\int_0^{1/2}
B(u)B(t-u)\; du$. This integral is less or equal than
\[
\int_0^{N^{-1}}
\frac{N\; du}{t-u}
+
\int_{N^{-1}}^{t-N^{-1}}
\frac{du}{u(t-u)}
+
\int_{t-N^{-1}}^{t+N^{-1}}
\frac{N\; du}{u}
+
\int_{t+N^{-1}}^{1/2+N^{-1}}
\frac{du}{u(u-t)},
\]
that gives $O\big(t^{-1}\log(Nt)\big)$ evaluating or estimating the integrals.
\end{proof}

\begin{proof}[Proof of Proposition~\ref{circular}]
Assume for convenience $-1/2 \leq x < 1/2$ and let $D_N^*(t)$ be the left hand side in \eqref{dir_b}. We have
\[
\sum_{0\le n\le N} r_{\alpha,\beta}(n) e(n x)
= \int_{-1/2}^{1/2}
\sum_{n \ge 0} r_{\alpha,\beta}(n) e\big(n (y+i/N)\big)
D_N^*(x-y)\; dy.
\]
By Lemma~\ref{b_dir} and Lemma~\ref{b_theta}
\begin{equation}\label{cirm}
\sum_{0\le n\le N} r_{\alpha,\beta}(n) e(n x)
\ll 
\sum_{a/q}
q^{-1}
\int_{\mathcal{A}_{a/q}}
B(y-a/q) B(x-y)\; dy
\end{equation}
where $\bigcup \mathcal{A}_{a/q}$ is the Farey dissection of $[-1/2, 1/2]$ of order 
$\lfloor N^{1/2}\rfloor$ as before. 
Trivially
\[
 \mathcal{I}_{a/q}
 :=
\int_{\mathcal{A}_{a/q}}
B(y-a/q) B(x-y)\; dy
\le (B*B)(x-a/q).
\]

If $a/q=a_x/q_x$ we employ Lemma~\ref{BcB} (with an extra logarithm to absorb an error term appearing later) to get 
\[
 \mathcal{I}_{a_x/q_x}\ll
 \frac{N(\log N)^2}{1+N|x-a_x/q_x|}.
\]
In the rest of the cases $|x-a/q|\gg N^{-1}$ is assured and 
$\mathcal{I}_{a/q}\ll |x-a/q|^{-1}\log N$ by
Lemma~\ref{BcB}.
Substituting in \eqref{cirm}
\begin{equation}\label{sum_x}
\sum_{0\le n\le N} r_{\alpha,\beta}(n) e(n x)
\ll 
 \frac{N(\log N)^2}{q_x+N|q_xx-a_x|}
+
(\log N)\sum_{a/q\ne a_x/q_x} |qx-a|^{-1}.
\end{equation}
Each summand attains its maximum when $x$ is one of the end-points of $\mathcal{A}_{a_x/q_x}$, both of which are rational numbers $A/Q$ with $Q \asymp N^{1/2}$. Hence doubling the sum, it suffices to bound 
\[
\sum_{a/q\ne a_x/q_x} 
|qA/Q-a|^{-1}
=
Q
\sum_{m\le N} 
m^{-1}
\#\big\{
a/q\; :\; Aq-Qa=\pm m
\big\}.
\]
The last cardinality is $O(1)$ and introducing this bound in \eqref{sum_x}, the result follows. 
\end{proof}

\section{Paraboloids}

We are ready to prove the first statement of Theorem~\ref{main} in the following stronger form:

\begin{theorem}\label{main2}
Let $\mathcal{P}$ be as in \eqref{main_P} with $d \geq 3$. Assume that the coefficient matrix $A = (a_{ij})$ of $Q$ satisfies $a_{12}/a_{11}, a_{22}/a_{11} \in \mathbb{Q}$. Then for each fixed $\epsilon>0$,
\[
 \mathcal{N}(R)
 =|\mathcal{P}|R^d
 +
 O\big(
 R^{d-2+\epsilon}
 \big)
\]
holds uniformly for $0<c\ll 1<R$ and $\vec{\beta}\in\R^{d-1}$.
\end{theorem}

The proof is divided in two steps: first we deal with the three-dimensional case where we can exploit the full rationality of $Q$, and then we extend the result to higher dimensions. The uniformity in $c$ and $\vec{\beta}$ is crucial for the second step to succeed, as these parameters have to be taken depending on $R$.

\begin{proof}[Proof of Theorem~\ref{main2}, case $d = 3$]
Rescaling $R$ and $c$ we may suppose that $Q$ is integral. We may also assume that the vector $(\alpha_1, \alpha_2) = R\vec{\beta}$ lies in $[0,1) \times [0, 1)$, since $\mathcal{N}(R)$ is $1$-periodic in these variables. Finally we assume $c>4R^{-2}$ because $\mathcal{N}(R)-|\mathcal{P}|R^3=O(R)$ when $c\ll R^{-1}$. 

We have
\begin{equation}\label{ma_er}
 \frac 12 \mathcal{N}(R)
 =
 \mathop{\sum\sideset{}{'}\sum}_{n_1,n_2}
 \Big(
 \lfloor
 f(n_1,n_2)
 \rfloor
 +\frac 12
 \Big)
 =
  \mathop{\sum\sideset{}{'}\sum}_{n_1,n_2}
 f(n_1,n_2)
 -
 \mathop{\sum\sideset{}{'}\sum}_{n_1,n_2}
 \psi\big(f(n_1,n_2)\big),
\end{equation}
where $f(x,y)= \big(cR^2-Q(x+\alpha_1,y+\alpha_2)\big)/R$, $\psi(x)=x-\lfloor x\rfloor -1/2$ and the prime indicates that the double summation is restricted to $Q(n_1+\alpha_1,n_2+\alpha_2)\le cR^2$. 

Let $\chi$ the characteristic function of $Q(x+\alpha_1,y+\alpha_2)\le cR^2$. Applying Euler-Maclaurin formula firstly in  $n_2$ and secondly in $n_1$, we have
\begin{eqnarray*}
  \mathop{\sum\sideset{}{'}\sum}_{n_1,n_2}
 f(n_1,n_2)
 &=&
 \sum_{|n_1|\ll R\sqrt{c}}
 \Big(
 \int \chi(n_1,y)f(n_1,y)\; dy 
 +
 O(1)
 \Big)
\\
 &=&
 \int \chi(x,y)f(x,y)\; dy dx 
 +
 O(R)
\end{eqnarray*}
and the last integral is, of course, $\frac 12 |\mathcal{P}|R^3$.  

It is well known (see for instance \cite{montgomery}) that for any $M\in\Z^+$ there exist trigonometric polynomials 
$Q^\pm(x)=\sum_{|m|\le M} a_m^\pm e(mx)$
such that $Q^-(x)\le \psi(x)\le Q^+(x)$
with $a_0^\pm \ll M^{-1}$ and $a_m^\pm \ll m^{-1}$.
Taking $M = \lfloor c^{1/2} R\rfloor$ we get from \eqref{ma_er} 
\begin{equation}\label{mel}
 \mathcal{N}(R)
 =
 |\mathcal{P}|R^3
 +
 O\big(\mathcal{E}(R)\big)+O\big(R^{1+\epsilon}\big)
\end{equation}
with
\[
 \mathcal{E}(R)
 =
 \sum_{m\le M}
 \frac 1m
 \Big|
 \mathop{\sum\sideset{}{''}\sum}_{n_1,n_2}
 e\Big(
 \frac{m}{R}Q(n_1+\alpha_1,n_2+\alpha_2)
 \Big)
 \Big|.
\]
The double prime indicates we have replaced the summation domain to $Q(n_1, n_2)\le M^2$, at the cost of adding and removing at most $O(M)$ terms.
Subdividing into dyadic intervals, there exists $K\le M$ such that
\[
 \mathcal{E}(R)
 \ll K^{-1} M^\epsilon
 \sum_{K\le m< 2K}
 \Big|
 \mathop{\sum\sideset{}{''}\sum}_{n_1,n_2}
 e\Big(
 \frac{m}{R}Q(n_1+\alpha_1,n_2+\alpha_2)
 \Big)
 \Big|.
\]
By Proposition~\ref{circular}, 
\begin{equation}\label{e_compl}
 \mathcal{E}(R)
 \ll K^{-1}
 M^{2 + \epsilon}
 \sum_{K\le m< 2K}
 \big(q+M^2|qm/R-a|\big)^{-1}
\end{equation}
where $a=a(m)$ and  $q=q(m)$ are determined by \eqref{farey_a} with $x=m/R$. In particular, we have
\[
 |mq-aR|\le \frac{R}{M}.
\]

If $K>R/M$ then $0\ne a \ll R$ and for each fixed $a$ we have that $m$ divides an integer in an interval of length $O(R/M)$. This leaves $O\big(R^{1+\epsilon}/M\big)$ possibilities for $m$. Neglecting the term $M^2|qm/R-a|$ in \eqref{e_compl} and using $q\asymp Ra/K$,
\[
 \mathcal{E}(R)
 \ll K^{-1}
 M^{2+\epsilon}
 \sum_{a \ll R}
 R^{1+\epsilon} M^{-1}
 (Ra/K)^{-1}
 \ll R^{1+\epsilon}.
\]

If $K \leq R/M$ the argument of divisibility fails when $a/q=0/1$. These terms can be estimated directly in \eqref{e_compl}, while the previous argument can be applied to those with $a \neq 0$, yielding again $\mathcal{E}(R) \ll R^{1+\epsilon}$.
\end{proof}

\begin{proof}[Proof of Theorem~\ref{main2}, case $d > 3$]
Write $\vec{x} = (\vec{x}_1, \vec{x}_2)$ and $\vec{\beta} = (\vec{\beta}_1, \vec{\beta}_2)$ 
with $\vec{x}_1,\vec{\beta}_1\in\R^2$
and
$\vec{x}_2,\vec{\beta}_2\in\R^{d-3}$.
Completing squares,
\begin{equation}\label{decompQ}
 Q(\vec{x} + \vec{\beta}) 
 = 
 Q_1(
 \vec{x}_1 + \vec{\gamma}
 ) 
 + 
 Q_2(\vec{x}_2 + \vec{\beta}_2),
\end{equation}
where $\vec{\gamma}$ depends linearly on $(\vec{x}_2,\vec{\beta}_1,\vec{\beta}_2)$ and the matrix of $Q_1$ is $(a_{ij})_{i,j=1}^2$.

Given $\vec{n}_2\in\Z^{d-3}$, let us denote by $\mathcal{P}_{\vec{n}_2}$ the 
three-dimensional slice of $\mathcal{P}$ obtained by fixing $\vec{x}_2=\vec{n}_2/R$, and by $\mathcal{N}_{\vec{n}_2}(R)$ the number of lattice points it contains after being dilated with scale factor $R$. By the three-dimensional case of this theorem and the decomposition \eqref{decompQ},
\[
 \mathcal{N}(R) 
 =
 \sum_{\vec{n}_2} 
 \mathcal{N}_{\vec{n}_2}(R)
 =
 \sum_{\vec{n}_2} 
 |\mathcal{P}_{\vec{n}_2}|R^3
 + O\big( R^{d - 2 + \epsilon}\big),
\]
both sums extended to the domain $Q_2(\vec{n}_2 + R\vec{\beta}_2) \leq cR^2$. A simple computation shows
\[
 |\mathcal{P}_{\vec{n}_2}| 
 = 
 \frac{ \pi}{ \sqrt{D}}
 \big(
 c - Q_2(\vec{n}_2/R + \vec{\beta}_2)
 \big)^2
\]
where $D$ is the determinant  of (the Hessian matrix of) $Q_1$. Applying the Euler-Maclaurin formula iteratively in one variable at a time we find 
\[
 \frac{ \pi}{ \sqrt{D}} 
 \sum_{\vec{n}_2}
 \big(
 c - Q_2(\vec{n}_2/R + \vec{\beta}_2)
 \big)^2 
 =
 \frac{ \pi}{ \sqrt{D}} 
 \int
 \big(
 c - Q_2(\vec{x}_2/R)
 \big)^2
 \; d\vec{x}_2
 +
 O\big(R^{d - 5}\big)
\]
and the main term in the right hand side is $|\mathcal{P}|R^{d-3}$.
\end{proof}

\section{Some $\Omega$-results}

Let us start considering first the two-dimensional parabolic region $\mathcal{P}_2$ introduced in \S2. The simple estimate $\mathcal{N}_2(R) - |\mathcal{P}_2|R^2 = \Omega\big(R^{1/2}\big)$, already contained in \cite{popov}, follows by noting that for some values of $R$ there are at least $\Omega\big(R^{1/2}\big)$ points lying on the boundary of $R\mathcal{P}_2$.\footnote{For clarity we denote throughout this section by $R\mathcal{K}$ the image of the set $\mathcal{K}$ under the homothety with respect to the origin and scale factor $R$.} Indeed, given any positive integer $M$, let $R = M^2$ and consider the points $\big(kM,\pm(M^2-k^2)\big)$ with $-M\le k\le M$. The following more precise one-sided $\Omega$-results are a consequence of Corollaries~\ref{prime4k1} and~\ref{classn}.

\begin{proposition}
The error term $\mathcal{E}(R) = \mathcal{N}_2(R)-|\mathcal{P}_2|R^2$ satisfies
\[
 \mathcal{E}(R)=\Omega_+\big(R^{1/2}\big)
 \quad
 \text{and}
 \quad
 \mathcal{E}(R)=\Omega_-\big(R^{1/2}\log\log R\big).
\]
\end{proposition}

\begin{proof}
The first statement follows by taking $N$ a square in Corollary~\ref{prime4k1}. For the second one we remark that the main result of \cite{BaChEr} asserts that there are infinitely many primes $p\equiv 3\pmod{4}$ satisfying $h(-p)/\sqrt{p}\gg \log\log p$. It suffices to take $N = p$ for any such prime $p$ in Corollary~\ref{classn}.
\end{proof}

The upper bound $h(-d)/\sqrt{d} \ll \log\log{d}$ is known to hold under the generalized Riemann hypothesis \cite{littlewood}. Any hope to obtain a better $\Omega_-$-result from Corollary~\ref{classn} therefore must take advantage of the sum of class numbers, and for this we need uniform lower bounds over certain families of discriminants. Fortunately Heath-Brown proved an astonishing result that, in some way, shows the absence of exceptional zeros for large multiples of some primes in a fixed set \cite{HeBr2}. Even more astonishing is the short and elementary proof of this fact. In its original form (see \cite{blomer} for an enhaced version with the same proof) the result claims that if $S$ is a fixed set of more than $505^2$ odd primes then for any sufficiently large integer $d$ there exists a prime $p_d \in S$ satisfying $L(1,\chi_{-p_dd}) \gg (\log d)^{-1/9}$. Using this we prove the following one-sided $\Omega$-result:

\begin{proposition}\label{omega2}
We have
\[
 \mathcal{N}_2(R)-|\mathcal{P}_2|R^2=
 \Omega_-\big(R^{1/2}\exp(c\sqrt{\log R}/\log\log R)\big)
 \qquad
 \text{for any }\quad c<\sqrt{2}.
\]
\end{proposition}

\begin{proof}
Let $S$ be the set of the first $505^2+1$ primes $p\equiv 3\pmod{4}$ and fix an integer $d_0$ large enough so that the aforementioned result of Heath-Brown holds for any $d \geq d_0$. Choose $N=N'\prod_{p\in S}p$ in Corollary~\ref{classn}, where $N'$ is the product of the primes $p\equiv 1\pmod{4}$ in the interval $[d_0,x]$ for any large $x$. Then by the class number formula,
\[
  \sum_{\substack{d\mid N\\ d\equiv3\ (4)}}
 \omega_d h(-d)
 \gg 
  \sum_{d\mid N'}
 \frac{\sqrt{p_dd}}{(\log d)^{1/9}}
 \gg
 \frac{\sqrt{N'}}{(\log N)^{1/9}}
 \prod_{p\mid N'}
 \big(1+p^{-1/2}\big).
\]
The result now follows by noting that the logarithm of the product over the primes is asymptotically $\sqrt{2\log N'}/\log\log N'$ and $N'\gg N$. 
\end{proof}

\

We now prove some $\Omega$-results  for higher dimensional centered rational paraboloids i.e., those of the form
\begin{equation}\label{ra_pa}
\mathcal{P} = \big\{
(\vec{x},y)\in\R^{d-1}\times \R
\;:\;
|y|\le c-Q(\vec{x})
\big\}
\qquad 
\text{with $c\in\Q$ and $Q$ rational}.
\end{equation}

\begin{theorem}\label{om_re}
The lattice point discrepancy
$\mathcal{N}(R)-|\mathcal{P}|R^d$ for $\mathcal{P}$ as in \eqref{ra_pa} 
is $\Omega\big(R^{d-2}\eta(R)\big)$, where
\[
  \eta(R) =
  \begin{cases}
  \exp\big(K\frac{\log R}{\log\log R}\big) & \text{for any } K < \log{2} \text{ when } d = 3, \\
  \log\log{R} & \text{when } d = 4, \\
  \sqrt{\log\log{R}} & \text{when } d = 5, \\
  1 & \text{when } d \geq 6.
  \end{cases}
\]
\end{theorem}

This proves that our main result is sharp in the sense that the $\epsilon$ in the exponent cannot be dropped in the low dimensional cases. Note that when $d \geq 6$ and $\mathcal{P}$ is as in \eqref{ra_pa} the lattice point discrepancy is actually $O\big(R^{d-2}\big)$, as shown by applying Euler-Maclaurin summation to the corresponding asymptotics for the number of lattice points in the dilated $(d-1)$-dimensional ellipsoid $\{Q(\vec{x}) \leq 1\}$ (see, for instance, \cite[\S21, Satz~1]{fricker}). For general paraboloids, however, our method does not provide an answer as to whether the $\epsilon$ is really necessary for $d \geq 6$.

\

To deduce Theorem~\ref{om_re} we consider $B(R)$ with $R\in \Z^+$, the number of lattice points on the boundary of $\mathcal{P}$ dilated by $R$. Clearly, an $\Omega$-result for $B(R)$ readily implies the same $\Omega$-result for the lattice point discrepancy. We will therefore prove Theorem~\ref{om_re} in the stronger form $B(R) = \Omega\big(R^{d-2}\eta(R)\big)$. Some reductions first: note that without loss of generality we may assume $c \in \mathbb{Z}$, and let $Q = \frac ab Q^*$ where $Q^*$ is a primitive integral quadratic form. For each $\vec{n}\in\Z^{d-2}$ with $Q^*(\vec{n})=Rn$ and $abn\le cR$ we have that the lattice point $(b\vec{n}, cR-abn)$ is counted by $B(R)$. In other words, 
\begin{equation}\label{bdrq}
B(R)
\ge
\sum_{n\le\alpha R}
r_{Q^*}(Rn)
\qquad\text{with}\quad
\alpha=\frac{c}{ab}
\end{equation}
where $r_{Q^*}(k)$ is the number of representations of $k$ by the quadratic form $Q^*$. For the remaining proofs we will not need to refer to $Q$ anymore, and therefore we will write $Q$ instead of $Q^*$ for the sake of notational simplicity.

\begin{proof}[Proof of Theorem~\ref{om_re}, case $d = 3$]
Let $r_1,r_2,\dots, r_k$ be the solutions of 
\[
 Q(r,1)\equiv 0\pmod{R}
\]
and for each $1\le j\le k$ and a fixed $0<\delta<1$ define
\[
 C_j=\big\{
 (x, y) \in \mathbb{Z}^2 \; : \; |y| \leq \delta R,\; |x| \leq \delta R,\;  
 x \equiv r_j y \pmod{R}
 \big\}.
\]
Choosing $\delta^2<\frac 14 \lambda^{-1}\alpha$ with $\lambda$ the greatest eigenvalue of the matrix of $Q$, we have  that $Q$ maps $C_j$ into multiples of $R$ less than $\alpha R^2/2$. Hence the sum in \eqref{bdrq} is at least $\#\bigcup_j C_j$. If we restrict  $y$ to $(y,R)=1$ then the sets $C_j$ become disjoint, consequently
\begin{equation}\label{b3au}
 B(R)
 \ge
 k\min_j \#C_j
 -
 k\, \#
 \big\{
 y \in \mathbb{Z} \; : \; |y| \leq  R,\; (y,R)>1
 \big\}.
\end{equation}
For each fixed $j$, consider the remainders of 
$0r_j$, $1r_j$, $2r_j$,\dots, $\lfloor \delta R\rfloor r_j$
when divided by $R$. 
By the pigeonhole principle, if we subdivide $[0,R)$ into $\lceil \delta^{-1}\rceil$ equal subintervals, at least $\delta R/\lceil \delta^{-1}\rceil$ of the remainders lie in the same subinterval. In this way, we have at least 
$\delta R/\lceil \delta^{-1}\rceil$
pairs $(u_\ell, v_\ell)$ such that $0\le v_\ell \le \delta R$ and $u_\ell\equiv r_jv_\ell$ is in a subinterval of length $R/\lceil \delta^{-1}\rceil$. Hence 
$(u_\ell-u_1,v_\ell-v_1)\in C_j$
and it follows $\#C_j\ge \delta R/\lceil \delta^{-1}\rceil$. In this way, \eqref{b3au} assures
\begin{equation}\label{b3au2}
 B(R)
 \ge
 k\frac{\delta^2 R}{{1}+\delta}
 +
 k\big(\varphi(R)-R\big).
\end{equation}
For large $x$, take $R$ as the product of the primes $x\le p\le 2x$ such that 
$\Big(\frac{4\Delta}{p}\Big)=1$ where $\Delta$ is the discriminant of $Q$. Again by the prime number theorem in arithmetic progressions, we have
\begin{equation}\label{pntr}
 \log R\sim \frac{x}{2}
 \qquad\text{and}\qquad
 \frac{\varphi(R)}{R}=
 \prod_{p\mid R}
 \big(1-p^{-1}\big)
 = 1+ O\Big(\frac{1}{\log x}\Big).
\end{equation}
The congruence $Q(r,1)\equiv 0$ admits two solutions modulo each of these primes $p$. Then by our choice  of $R$ we have that $k$ equals $2$ to the number of such primes that is at least $(\log R)/\log(2x)$. Substituting this and \eqref{pntr} in \eqref{b3au2}, we get the expected result. 
\end{proof}

\begin{proof}[Proof of Theorem~\ref{om_re}, case $d = 4$]
Combining Theorem~1 of \cite{blomer} and Theorem~2 of \cite{DuSP} we have
\begin{equation}\label{rrgen}
r_Q(n)
=
r_Q^{\textrm{gen}}(n)
+
O\big(n^{13/28+\epsilon}\big)
\qquad\text{for $n\not\in\mathcal{S}$}
\end{equation}
where $\mathcal{S}$ is a finite union of sets of the form $\{t_jm^2 \,:\, m \in \Z\}$ for some $t_j \in \Z$. 
Here $r_Q^{\textrm{gen}}$ is the average number of representations by forms belonging to the same genus as $Q$ that can be computed with Siegel mass formula (see \cite[\S20.4]{IwKo} for the definitions and details). In Lemma~6 of \cite{ChCrUb} this formula was written as
\begin{equation}\label{rgen}
r_Q^{\textrm{gen}}(n)
=
\frac{4\pi\sqrt{2n}}{\sqrt{D}}
\sum_{d^2\mid n}
d^{-1}{U}(n/d^2)L(1,\chi_{-2Dn/d^2})
\end{equation}
where $D$ is the determinant of $Q$, $L$ is the $L$-function corresponding to the Kronecker symbol $\chi_m$ modulo $m=-2Dn/d^2$ and $U$ is a certain $8D^2$-periodic function  which is nonnegative and not identically zero. 

Let us say $(R,2D)=1$ and for each $d^2\mid R$ choose $n_d$ such that $U(n_dR/d^2)\ne 0$, then \eqref{rrgen} and \eqref{rgen} together with \eqref{rqm} imply 
\begin{equation}\label{rqm}
B(R)
\gg
R
\sum_{d^2\mid R}
d^{-1}\mathcal{L}_d(R)
+
O\big(R^{27/14+\epsilon}\big)
\end{equation}
where 
\[
 \mathcal{L}_d(R)
 =
 \sum_{n\in\mathcal{A}}L(1,\chi_{-2DRn/d^2})
 \qquad\text{with}\quad
 \mathcal{A}
 =
 \big\{
 n\ll R\;:\; n\not\in\mathcal{S},\ n\equiv n_d\pmod{8D^2}
 \big\}.
\]
If $\mathcal{L}_d(R)\gg R$, choosing $R=\prod_{2D<p\le x} p^2$ we have $\log R\sim 2x$ and
\[
B(R)
\gg
R^2
\prod_{2D<p\le x} \big(1+p^{-1}\big)
+
O\big(R^{27/14+\epsilon}\big)
\gg R^2\log\log R.
\]

It remains to prove $\mathcal{L}_d(R)\gg R$. Expanding the $L$-functions, we can write $\mathcal{L}_d(R)$ as
\[
S_1 + S_2 + S_3 :=
\sum_{m_1} \frac 1{m_1} 
\sum_{n\in\mathcal{A}} \chi_{d_n}(m_1)
+ 
\sum_{m_2} \frac{\chi_{-2DR'}(m_2)}{m_2} 
\sum_{n\in\mathcal{A}} \chi_n(m_2)
+ 
\sum_{n\in\mathcal{A}} 
\sum_{m_3} \frac{\chi_{d_n}(m_3)}{m_3}
\]
where $d_n = -2DR'n$, $R' = R/d^2$, 
$m_1$ runs over the squares in $[1,R^{1+\epsilon}]$, $m_2$ over the non-squares coprime to $2DR'$ in the same interval and $m_3>R^{1+\epsilon}$. 
Trivially, $S_1\gg R$. By P\'olya-Vinogradov inequality $S_3\ll \sum_{n\in\mathcal{A}}R^{-\epsilon}\ll R^{1-\epsilon}$.
There are $O(R^{1/2})$ values of $n\ll R$ in $\mathcal{S}$  that when added to $\mathcal{A}$ give a negligible contribution $O(R^{1/2}\log R)$ to $S_2$, and hence we can drop the condition $n\not\in \mathcal{S}$ in $S_2$. On the other hand, the congruence condition $n\equiv n_d$ can be detected inserting 
$\sum_\chi \chi(n)\overline{\chi}(n_d)/\varphi(8D^2)$ where $\chi$ runs over the characters modulo $8D^2$. Since $\gcd(m_2,2DR')=1$, $\psi(n)=\chi(n)\chi_{n}(m_2)$ is a nonprincipal character modulo $8D^2m_2$ and P\'olya-Vinogradov inequality proves $S_2\ll R^{1/2+\epsilon}$. Therefore $\mathcal{L}_d(R)\sim S_1\gg R$.
\end{proof}

\begin{proof}[Proof of Theorem~\ref{om_re}, case $d \geq 5$]
For $d \geq 6$ we have by Corollary~11.3 of \cite{iwaniec} the estimate $r_{Q}(m) \asymp m^{(d-3)/2}$ as long as $m$ is sufficiently large and $Q(\vec{x}) \equiv m$ is solvable modulo $2^7 D^3$ with $D$ the determinant of $Q$. Taking $m = Rn$ with $R$ a large multiple of $2^7 D^3$, both conditions are fulfilled   and the result follows from  \eqref{bdrq}.

If $d = 5$, Corollary~11.3 of \cite{iwaniec} gives for $2^7D^3\mid R$
\begin{equation}\label{b5au}
 B(R)
 \gg 
 R
 \sum_{n\ll R}
 n
 \prod_{p\mid Rn}
 \big(1+\chi_D(p)p^{-1}\big)
 \qquad\text{with }\quad 
 \chi_D(p)=\Big(\frac{D}{p}\Big).
\end{equation}

Let  $P_D$ the product of the primes $p\le x$ such that $\chi_D(p)=1$. By the prime number theorem in arithmetic progressions, we have
\[
 \log P_D\sim \frac{x}{2}
 \qquad\text{and}\qquad
 \prod_{p \mid P_D}
 \big(1+p^{-1}\big)
 \gg \sqrt{\log x}
 \sim\sqrt{\log\log P_D}.
\]
Choosing $R=2^7D^3P_D$ in \eqref{b5au}, we have
\[
 B(R)
 \gg
 R
 \prod_{p\mid P_D}
 \big(1+p^{-1}\big)
 \cdot
 \sum_{n\ll R} n
 \prod_{p\mid n}
 \big(1-p^{-1}\big).
\]
The sum equals that of $\varphi(n)$ (Euler's totient function) that is comparable to $R^2$.
\end{proof}

\textbf{Acknowledgement.} We are grateful to A. D. Mart\'inez for his useful comments on an early version of this article.

\bibliographystyle{plain}

\end{document}